\documentclass{amsart}
\usepackage{amsmath}
\usepackage{amssymb}
\usepackage{color}
\newtheorem{theorem}{Theorem}[section]
\newtheorem{corollary}[theorem]{Corollary}
\newtheorem{lemma}[theorem]{Lemma}

\theoremstyle{definition}
\newtheorem{definition}[theorem]{Definition}
\newtheorem{remark}[theorem]{Remark}

\numberwithin{equation}{section}

\begin{document}

\title[Convex characteristics of quaternionic positive definite functions]{Convex characteristics of quaternionic positive definite functions on abelian groups}

\thanks{}
\author[Z. P. Zhu]{Zeping Zhu}
\email{zzp$\symbol{64}$mail.ustc.edu.cn}


\begin{abstract}
This paper is concerned with the topological space of normalized quaternion-valued positive definite functions on an arbitrary abelian group $G$, especially its convex characteristics. There are two main results. Firstly, we prove that the extreme elements in the family of such functions are exactly the homomorphisms from $G$ to the sphere group $\mathbb{S}$, i.e., the unit $3$-sphere in the quaternion algebra. Secondly, we reveal a phenomenon which does not exist in the complex setting: The compact convex set of such functions is not a Bauer simplex except when $G$ is of exponent $\leq 2$. In contrast, its complex counterpart is always a Bauer simplex, as is well known. We also present an integral representation for such functions as an application and some other minor interesting results.
\end{abstract}
\maketitle

\section{Introduction}
\subsection{background}

The theory of positive definite and related functions is an important ingredient in harmonic analysis with strong connections to moment problems, unitary representations, reproducing kernel Hilbert spaces, stationary stochastic processes, Hoeffding-type inequalities and many other fields. Our interest lies on its further development in the noncommutative setting, especially the quaternionic case.

It is well known that under certain regularity conditions every complex-valued positive definite function is the Fourier-type transform of a positive measure. This is formulated in exact terms in the famous theorems of Herglotz, Bochner, Hamburger and Bernstein-Widder. All these theorems are special cases of a theorem in the general setting of abelian semigroups (one may refer to, e.g., \cite{Berg-1976,Berg-1984,Lindahl-1971} for details). The four results mentioned above correspond to the following semigroups with involution respectively: $$(\mathbb{Z},+,x^*=-x),\ (\mathbb{R},+,x^*=-x),\  (\mathbb{N},+,x^*=x),\ ([0.+\infty),+,x^*=x).$$ These contributions compose a rather satisfactory theory of the characterization of complex-valued positive definite functions. Some earlier contributions on  this topic or related topics in the quaternionic setting have been made by Alpay along with his coauthors (see, e.g., \cite{Alpay-2016,Alpay-2017}) and some other academics. But there are still some fundamental problems that remain unsolved.

\subsection{study subject}
In the complex setting, for a normalized positive-definite function $\phi$ on an abelian group $G$, there exists a unique regular Borel probability measure $\mu$  on the dual group $\widehat{G}$ such that
$$\phi(x)=\int_{\widehat{G}}\gamma(x)d\mu(\gamma), $$
and vice verse.
Here the dual group $\widehat{G}$ consists of complex characters on $G$. From the view-point of convexity, this correspondence  indicates that the extreme elements in the convex set $\mathcal{P}_*$ of normalized positive-definite functions are precisely the complex characters. Moreover, based on Choquet's theory, the uniqueness of $\mu$ implies that $\mathcal{P}_*$ is a simplex, thereby a Bauer simplex since its extreme boundary is closed (one may refer to Theorems 3.6 and 4.1 in Chapter II of \cite{Alfsen-1971} for more details). In brief, $\mathcal{P}_*$ shares almost the same convex characteristics with the finite dimensional simplex. These results are also valid for the exponentially bounded positive definite functions on abelian semigroups (see, e.g., \cite{Berg-1976,Berg-1984}).

Although, there are several generalizations of Bochner's theorem that characterize the quaternionic positive definiteness, the convex structure of
normalized quaternion-valued positive definite functions is still unrevealed, since the integral characterizations are given by the Fourier-type transforms of certain kinds of quaternion-valued measures, not necessarily positive ones, on the dual group (see, e.g., Theorem 4.5 in \cite{Alpay-2016}). In this paper, we intend to solve two questions on this topic:
\begin{itemize}
  \item[1)]How to characterize the extreme boundary of the convex set of normalized quaternion-valued positive definite functions on an abelian group.

  \item[2)]Whether this convex set is a Bauer simplex in the pointwise convergence topology.
\end{itemize}

\subsection{main results}
The answers to the main questions mentioned above are as follows:
\begin{itemize}
  \item[1)]The extreme boundary of the convex set of normalized quaternion-valued positive definite functions consists of all the quaternion-valued characters.

  \item[2)]This set is a Bauer simplex if and only if the given group is of exponent $\leq 2$, or equivalently it is a $\mathbb{Z}_2$-vector space.
\end{itemize}
A global characterization of quaternion-valued positive definite functions follows naturally:
\begin{itemize}
  \item[3)]for a normalized positive-definite function $\phi$ on an abelian group $G$, there exists a (not necessarily unique)  regular Borel probability measure $\mu$  on its quaternionic dual $G^\delta$ such that
$$\phi(x)=\int_{G^\delta}\Gamma(x)d\mu(\Gamma). $$
Conversely, the quaternionic positive definiteness is characteristic for such expressions.
\end{itemize}
Here $G^\delta$ consists of all the quaternion-valued characters, i.e., homomorphisms from $G$ to the group of unit quaternions with the composition being the multiplication inherited from the quaternion algebra.
Intuitively speaking, in most cases (except when $G$ is of exponent $\leq 2$) the convex set of normalized quaternion-valued positive definite functions has a more evenly distributed extreme boundary in contrast to its complex counterpart, the extreme boundary of which is less evenly distributed as in the primitive case of a $n$-simplex.

We can say that our work reveals a phenomenon different from the classical case, and also fills a research gap in the quaternionic setting.

\subsection{arrangement of sections}
This paper is organized as follows. Section 2 contains basic notions and properties of quaternion-valued positive definite functions. Section 3 is split into three major parts. In the first subsection, the main questions are discussed in the special case of $G=\mathbb{Z}$. In the second subsection, the first one of the main results above is established in the general case when $G$ is an arbitrary abelian group, and an integral representation for quaternion-valued positive definite functions, i.e., the third main result, is given as an application. The final subsection is devoted to the second one of the main results above.


\section{definitions and basic properties}
We work mainly on the quaternion-valued positive definite functions on abelian groups. Nevertheless, the relevant concepts shall be introduced in the more general setting of semigroups with involution.

\subsection{quaternionic positive definite functions}\label{subsec-pdf}

\begin{definition}\cite{Drazin-1978}
  Let $S$ be a semigroup, i.e., a nonempty set furnished with an associative binary operation $\circ$ and a neutral element $e$. An involution on $S$ is a bijection $s\mapsto s^*$ of $S$ onto itself, satisfying
  $(\alpha^*)^*=\alpha,\ (\alpha\beta)^*=\beta^*\alpha^*$.
\end{definition}

For an abelian group, the composition and the neutral element are conventionally denoted by $+$ and $0$; and there are two natural involutions defined as $s^*=-s$ and $s^*=s$, respectively.

Let $\mathbb{H}$ denote the real quaternion algebra
$$\{q=a_0i_0+a_1i_1+a_2i_2+a_3i_3:\ a_i\in\mathbb{R}\ (i=0,1,2,3)\},
$$
where $i_0=1$, $i_3=i_1i_2$, and $i_1,i_2$ are the generators of  $\mathbb H$, subject to the following identities:
 $$i_1^2=i_2^2=-1, \qquad i_1i_2=-i_2i_1. $$
For any $q\in \mathbb{H}$,
its conjugate is defined as $\overline{q}:=a_0i_0-a_1i_1-a_2i_2-a_3i_3$, and its norm given by $\lvert q\rvert:=\sqrt{a_0^2+a_1^2+a_2^2+a_3^2}$. One may refer to \cite{Gurlebeck-2008} for more details about the real quaternion algebra.

Hereinafter we always assume that $S$ is a semigroup equipped with some involution.
\begin{definition}
  A quaternion-valued function $\phi$ on $S$ is said to be  positive definite if for any $s_1,s_2,\cdots,s_k\in S$  and any $q_1,q_2,\cdots,q_k\in\mathbb{H}$, the following inequality \begin{equation}\label{Eq-Def-PDF}
    \sum_{1\leq i,j\leq k} \overline{q_i}\phi(s_i^*\circ s_j)q_j\geq 0
  \end{equation}
  is satisfied.
\end{definition}

We denote the set of quaternion-valued  positive definite functions on $S$ by $\mathcal{P}^\mathbb{H}(S)$. Additionally, a function $\phi$ on $S$ is called normalized if $\phi(e)=1$. Then we define $$\mathcal{P}_*^\mathbb{H}(S):=\{\phi\in\mathcal{P}^\mathbb{H}(S): \phi \text{ is normalized}\}. $$
One may refer to \cite{Alpay-2016,Alpay-2017}, etc for analogous definitions of quaternionic positive definiteness.

\begin{theorem}\label{thm-basic-properties}
  Any quaternion-valued  positive definite function $\phi$ satisfies the following properties:
  \begin{itemize}
  \item[i)]$\phi$ is hermitian, i.e., $\phi(s^*)=\overline{\phi(s)}$ for $s\in S$;
  \item[ii)]$\phi(\alpha^*\circ \alpha)\geq0$ and $\lvert\phi(\alpha^*\circ \beta)\rvert^2\leq\phi(\alpha^*\circ \alpha) \phi(\beta^*\circ \beta)$.
  \end{itemize}
\end{theorem}
The proof of this theorem is omitted since it follows the same procedure as in the classical case.

\begin{theorem}\label{thm-positive-definiteness}
Let $\phi$ be a quaternion-valued  positive definite function on $S$. Then for any $p_i\in\mathbb{H}$ and $t_i\in S$ $(i=1,2,\cdots,n)$, the function $\varphi$ defined as
$$\varphi(s):= \sum_{1\leq j,k\leq n} \overline{p_j}\phi(t_j^*\circ s \circ t_k)p_k \quad\text{for}\quad s\in S,$$
is also positive definite on $S$.
\end{theorem}
\begin{proof}
For any $s_1,s_2,\cdots,s_m\in S$ and any $q_1,q_2,\cdots,q_m\in\mathbb{H}$, a direct calculation yields
$$\sum_{1\leq j',k'\leq m} \overline{q_{j'}}\varphi(s_{j'}^*\circ s_{k'})q_{k'}=\sum_{1\leq j',k'\leq m}\sum_{1\leq j,k\leq n}\overline{p_jq_{j'}}\phi\left((s_{j'}\circ t_j)^* \circ (s_{k'}\circ t_k)\right)p_kq_{k'}; $$
thus we derive by the quaternionic positive definiteness of $\phi$ that
$$\sum_{1\leq j',k'\leq m} \overline{q_{j'}}\varphi(s_{j'}^*\circ s_{k'})q_{k'} \geq 0. $$
It completes the proof.
\end{proof}

\subsection{exponential boundedness}

\begin{definition}\label{def-absolute-value}\cite{Berg-1984}
  A nonnegative real valued function $\alpha$ on $S$ is called an absolute value if it satisfies the following conditions:
  \begin{itemize}
  \item[i)]$\alpha(e)=1$;
  \item[ii)]$\alpha(s\circ t)\leq \alpha(s)\alpha(t)$;
  \item[iii)]$\alpha(s^*)=\alpha(s)$.
  \end{itemize}
\end{definition}
As shown in the complex case, this notion provides a certain type of boundedness condition that plays a vital role in the spectral characterization of positive definite functions.

\begin{definition}
A quaternion-valued function $\phi$ on $S$ is called bounded w.r.t. an absolute value $\alpha$ (or $\alpha$-bounded for short) if $\phi$ is dominated by $\alpha$, i.e., there exists a positive real constant $C$ such that
$$\lvert\phi(s)\rvert\leq C\alpha(s) \quad\text{for}\quad s\in S. $$
\end{definition}
For a fixed absolute value $\alpha$, we write the set of $\alpha$-bounded quaternioin-valued positive definite functions on $S$ as $\mathcal{P}_\alpha^\mathbb{H}(S)$, and the set of normalized elements in $\mathcal{P}_\alpha^\mathbb{H}(S)$ as $\mathcal{P}_{\alpha,*}^\mathbb{H}(S)$.

We say a function $\phi$ is exponentially bounded if it is dominated by some absolute value. The set of exponentially bounded quaternioin-valued positive definite functions on $S$ is denoted by $\mathcal{P}_e^\mathbb{H}(S)$, and
the set of normalized elements in $\mathcal{P}_e^\mathbb{H}(S)$ by $\mathcal{P}_{e,*}^\mathbb{H}(S)$.

\begin{theorem}\label{thm-inequality}
  Let $\phi$ be a quaternioin-valued positive definite function on $S$ and bounded w.r.t. an absolute value $\alpha$. Then
  $$|\phi(s)|\leq\phi(e)\alpha(s)\quad \text{for} \quad s\in S. $$
\end{theorem}
The proof of this result is quite similar to that given in the complex case \cite{Berg-1984} and so is omitted.

A simple corollary follows immediately by applying Tychonoff's theorem.
\begin{corollary}\label{cor-compactness}
For any absolute value $\alpha$ on $S$, the (possibly empty) convex set $\mathcal{P}_{\alpha,*}^\mathbb{H}(S)$ is  compact in the pointwise convergence topology.
\end{corollary}

\subsection{RKHS associated with $\phi$}
Assume that $S$ is a semigroup  with involution  and the function $\phi:S\mapsto \mathbb{H}$ is positive definite. Denote the set of all quaternion-valued functions on $S$ by $\mathbb{H}^S$. We would like to mention that we only focus on the conventional right quaternion-linear structure in $\mathbb{H}^S$ given by the pointwise scalar multiplication on functions from the right side, which is to say that for $f\in \mathbb{H}^S$ and $q\in\mathbb{H}$, $fq$ is defined by
$$(fq)(s)=f(s)q \quad \text{for} \quad s\in S. $$
In the quaternionic setting, due to some technical reasons a right quaternion-linear space may be endowed with an extra well-chosen left linear structure (see, e.g., \cite{Alpay-2016-1,Alpay-2016-2}). Sometimes, the situation is different from what is usual in the commutative case. So one should be careful when considering the both-side quaternionic linearity.

Let us consider the right quaternion-linear subspace $H_\phi$ of $\mathbb{H}^S$ generated by the functions $\{\phi_s\}_{s\in S}$ with $\phi_s(t):=\phi(t^*\circ s)$. We furnish $H_\phi$ with a quaternion-valued scalar product defined as
$$\langle\sum_{s}\phi_s q_s, \sum_{t}\phi_t p_t\rangle:= \sum_{s,t} \overline{p_t}\phi(t^*\circ s) q_s$$
where the indices $s$ and $t$ run over two arbitrary finite nonempty subsets of $S$ respectively, and $q_s, p_t\in\mathbb{H}$. It is obvious that this product satisfies
 \begin{eqnarray*}
  \langle f, g\rangle&=& \overline{\langle g, f\rangle}, \\
  \langle f+g,h\rangle&=&\langle f,h\rangle+\langle g,h\rangle, \\
  \langle fp, g\rangle&=&\langle f, g\rangle p, \\
  \langle f, gp\rangle&=&\overline{p}\langle f, g\rangle, \\
 \end{eqnarray*}
for all $f, g, h\in H_\phi$ and $p\in\mathbb{H}$.
In addition, since the function $\phi$ is positive definite, the inequality  $\langle f, f\rangle\geq0$ always holds. Furthermore, the Cauchy-Schwarz inequality yields that for any $f=\sum_{s}\phi_s q_s\in H_\phi$, and $t\in S$,
$$\lvert f(t)\rvert= \lvert\sum_{s}\phi(t^*\circ s) q_s\rvert=\lvert\langle f, p_t\rangle\rvert\leq \langle f,f\rangle^{1/2}\langle \phi_t,\phi_t\rangle^{1/2}, $$ which implies
$$\langle f, f\rangle=0 \quad\text{iff}\quad f=0. $$
Therefore, $\langle\cdot, \cdot\rangle: H_\phi\times H_\phi\mapsto\mathbb{H}$ is a quaternionic inner product (see, e.g., \cite{Alpay-2016-1,Viswanath-1971}) on $H_\phi$. We call the completion of $H_\phi$ as the reproducing kernel Hilbert space (RKHS) associated with $\phi$, denoted by $\mathcal{H}_\phi$. Because all the quaternion-valued functionals $\Lambda_s: f \mapsto f(s)$ with $s\in S$ are continuous on $H_\phi$, the norm topology on $H_\phi$ must be stronger than or equal to the pointwise convergence topology. Accordingly, its completion $\mathcal{H}_\phi$ can be treated as a subset of  $\mathbb{H}^S$, i.e., the space of all quaternion-valued functions on $S$. In other words, every element in the completion is a concrete function as well.

There is an alternative way to construct the RKHS:
Let $H$ be the family of quaternion-valued functions on $S$ with finite supports. Obviously, the positive definite function $\varphi$ can induce a (possibly degenerate) inner product:
 $$\langle h,k\rangle:=\sum_{s,t\in S}\overline{k(t)}\varphi(t^*\circ s)h(s), $$
for all $h,k\in H$. Quotienting  $H$ by the subspace of functions with zero norm  eliminates the degeneracy. Then taking the completion gives a quaternionic Hilbert space which is isomorphic to $\mathcal{H}_\phi$ via the mapping $h\mapsto f=\sum_{s}\phi_sh(s)$.

A natural representation of $S$  on the quaternionic pre-Hilbert space $H_\phi$ is given as
\begin{equation}\label{eq-unitary-rep}
  \lambda(s)\left(\sum_{t}\phi_t q_t\right):=\sum_{t}\phi_{s\circ t} q_t \quad\text{for}\quad s\in S,
\end{equation}
where the index $t$ runs over an arbitrary nonempty subset of $S$, and $q_t$ is a quaternion for every $t$. Theoretically, this representation contains all the information about $\phi$. One can easily observe that
\begin{equation}\label{eq-representation}
  \phi(s)=\langle \lambda(s)\phi_e,\phi_e\rangle \quad\text{for}\quad s\in S.
\end{equation}
Moreover, the norm-boundedness of $\lambda$ is equivalent with the exponential boundedness of $\phi$ as shown below.

\begin{theorem}\label{thm-boundedness}
  Let $\phi$ ba a quaternion-valued positive definite function on $S$, and $\lambda$ be the representation as above. Then $\phi$ is exponentially bounded if and only if $\lambda(s)$ is a bounded operator on $H_\phi$ for all $s\in S$.
\end{theorem}
\begin{proof}First, we show  the sufficiency. Assume $\lambda(s)$ is a bounded operator on $H_\phi$ for all $s\in S$. It is easy to verify that $\alpha(s):=\|\lambda(s)\|$ is an absolute value. And by \eqref{eq-representation} we obtain
$$|\phi(s)|\leq\|\phi_e\|^2\|\lambda(s)\|. $$
Hence, $\phi$ is bounded w.r.t. $\alpha$, thereby exponentially bounded.

Next, we prove the necessity. Suppose that $\phi$ is exponentially bounded. Then there exists an absolute value  $\alpha$ which dominates $\phi$. Take an arbitrary function $f=\sum_{i=1}^{n}\phi_{t_i}p_i\in H_\phi$, and define
$$\varphi(s):= \sum_{1\leq j,k\leq n} \overline{p_j}\phi(t_j^*\circ s \circ t_k)p_k \quad\text{for}\quad s\in S.$$

Firstly, Theorem \ref{thm-positive-definiteness} tells $\varphi$ must be positive definite.
Secondly, Theorem \ref{thm-inequality}, together with the last two conditions in Definition \ref{def-absolute-value}, yields the inequality  $|\phi(t_j^*\circ s \circ t_k)|\leq \phi(e)\alpha(t_j)\alpha(t_k)\alpha(s)$. We thus derive that
$$|\varphi(s)|\leq \left(\phi(e)\sum_{1\leq j,k\leq n}|p_j||p_k|\alpha(t_j)\alpha(t_k)\right)\alpha(s) $$
holds for all  $s\in S$, which means $\varphi$ is also bounded w.r.t. $\alpha$.
Then applying Theorem \ref{thm-inequality} again  we have
$$|\varphi(s)|\leq \varphi(e)\alpha(s)=\|f\|^2\alpha(s). $$
It follows that
$$\|\lambda(s)f\|^2= \langle\lambda(s^*\circ s)f,f\rangle=\varphi(s^*\circ s)\leq \|f\|^2\alpha(s^*\circ s)\leq\|f\|^2\alpha(s)^2. $$
Therefore, $\lambda(s)$ is bounded for all $s$.
\end{proof}

Based on the result above, we know that when $\phi$ is exponentially bounded, the representation $\lambda$ has a unique bounded extension to the quaternionic Hilbert space $\mathcal{H}_\phi$. More explicitly, the extension of $\lambda(s)$ coincides with the left shifting operator $L_{s^*}$ given as
$$L_{s^*}(f)(t):=f(s^*\circ t),$$
for all $f\in \mathcal{H}_\phi$ and $t\in S$.

\subsection{relations between complex and quaternionic positive definiteness}
It is not difficult to show that a complex positive definite function must be quaternionic positive definite.
\begin{theorem}\label{thm-equiv-c-h-pd}
  Let $\phi$ be a $\mathbb{C}_I$-valued function on $S$. Then $\phi$ is complex positive definite if and only if it is quaternionic positive definite.
\end{theorem}
Here $\mathbb{C}_I$ stands for an arbitrary complex slice of $\mathbb{H}$, namely $\mathbb{C}_I:=\mathbb{R}\oplus I\mathbb{R}$ with $I$ being an imaginary unit in $\mathbb{H}$ \cite{Gurlebeck-2008}.
\begin{proof}
Evidently, the quaternionic positive definiteness is stronger than the complex positive definiteness. Thus we only need to  verify the necessity.

Suppose that $\phi: S\mapsto\mathbb{C}_I$ is complex positive definite, which is to say that for any $c_1,\cdots,c_n\in\mathbb{C}_I$ and $s_1,\cdots,s_n\in\mathbb{C}_I$, the following inequality holds:
\begin{equation}\label{eq-proof-thm-equiv-c-h-pd}
  \sum_{i,j=1}^{n}\overline{c_i}c_j\phi(s_i^*\circ s_j)\geq 0.
\end{equation}
Take an imaginary unit $J$ in the quaternion algebra such that $I\perp J$, or equivalently $IJ=-JI$. Notice that any $q\in\mathbb{H}$ can be expressed as
$$q=c+c'J\quad\text{with}\quad c,c'\in\mathbb{C}_I.$$
Hence, for any $q_1,\cdots,q_n\in\mathbb{H}$ and $s_1,\cdots,s_n\in\mathbb{C}_I$, we can decompose the summary $\sum_{i,j=1}^{n}\overline{q_i}\phi(s_i^*\circ s_j)q_j$ as
$$\sum_{i,j=1}^{n}\overline{c_i}c_j\phi(s_i^*\circ s_j)+\sum_{i,j=1}^{n}\overline{c'_i}c'_j\phi(s_i^*\circ s_j)-J\sum_{i,j=1}^{n}\overline{c'_i}c_j\phi(s_i^*\circ s_j)+J\sum_{i,j=1}^{n}\overline{c'_j}c_i\overline{\phi(s_i^*\circ s_j)}, $$
where $q_i=c_i+c'_iJ$ with $c_i, c'_i\in\mathbb{C}_I$ $(i=1,\cdots,n)$.
\eqref{eq-proof-thm-equiv-c-h-pd} indicates the first two terms are nonnegative. And the last two terms eliminate each other since $\phi$ is hermitian, i.e., $\phi(s^*)=\overline{\phi(s)}$. Consequently, we obtain
$$\sum_{i,j=1}^{n}\overline{q_i}\phi(s_i^*\circ s_j)q_j\geq 0,$$
which means $\phi$ is quaternionic positive definite.
\end{proof}

As mentioned in the proof above, any $q\in\mathbb{H}$ can be expressed as
$$q=c+c'J\quad\text{with}\quad c,c'\in\mathbb{C}_I,$$ where $J$ is an arbitrary imaginary unit vertical to $I$. Moreover, the value of $c$ is irrelative with the choice of $J$. We call $c$ the projection of $q$ on the complex slice $\mathbb{C}_I$. For convenience, we denote the mapping $q\mapsto c$ by $P_I$.

\begin{theorem}\label{Thm-qpd-to-cpd}
  Let $\phi$ be a quaternion-valued positive definite function on $S$. Then $P_I\phi$, i.e., the projection of $\phi$ on $\mathbb{C}_I$ is complex positive definite  for any imaginary unit $I\in\mathbb{H}$.
\end{theorem}
\begin{proof}
Note that the equality
$$P_I(q)=\frac{q-IqI}{2}$$
is valid for any quaternion $q$, and any imaginary unit $I$. Thus for any  $c_1,\cdots,c_n\in\mathbb{C}_I$ and $s_1,\cdots,s_n\in\mathbb{C}_I$, the summation
$$
\sum_{i,j=1}^{n}\overline{c_i}c_jP_I\phi(s_i^*\circ s_j)
$$
can be split into two parts
$$
1/2 \sum_{i,j=1}^{n}\overline{c_i}\phi(s_i^*\circ s_j)c_j+1/2 \sum_{i,j=1}^{n}\overline{c_iI}\phi(s_i^*\circ s_j)c_jI, \\
$$
and thereby is nonnegative  owing to  the quaternionic positive definiteness of $\phi$.
In conclusion, $P_I\phi$ is complex positive definite for any imaginary unit $I\in\mathbb{H}$.
\end{proof}
The reverse of the theorem above is not true, that is, $\phi:S\mapsto\mathbb{H}$  need not necessarily be positive definite, when $P_I\phi:S\mapsto\mathbb{C}_I$ is positive definite for all $I$.

\section{Characterizations for quaternion-valued positive definite functions}
As mentioned at the beginning of Subsection \ref{subsec-pdf},
On an abelian group, there exist two natural involutions, the identical involution $s^*=s$ and the inverse involution $s^*=-s$. Correspondingly, we get two different types of positive definiteness for quaternion-valued functions on abelian groups. One is called positive definiteness in the semigroup sense, applying to the case when the group is endowed with the identical involution. The other is called positive definiteness in the group sense, which is suitable when the inverse involution is equipped.

If a quaternion-valued function $\phi$ is positive definite in the semigroup sense, then Theorem \ref{thm-basic-properties} yields that $\phi(s)=\phi(s^*)=\overline{\phi(s)}$ is valid on the whole group. So the function $\phi$ must be real valued. It means there is essentially no such thing as a quaternion-valued positive definite function in the semigroup sense. Besides, there already exists a quite complete theory about the real valued positive definite functions on abelian groups. For these reasons, we only focus on the positive definiteness in the group sense.
And all the positive definite functions hereinafter appearing, by default, are positive definite in the group sense.

Given an arbitrary quaternion-valued positive definite function $\phi$ on an abelian group $G$, it is not difficult to verify that $\rvert\phi(g)\lvert\leq \phi(0)$ holds for all $g\in G$, which is a direct generalization  of the boundedness of complex-valued positive definite functions on abelian groups to the quaternionic case via Theorem \ref{Thm-qpd-to-cpd}. Thus we see every quaternion-valued positive definite function on an abelian group is dominated by the constant absolute value $1$, thereby exponentially bounded.

\subsection{characterizations for quaternion-valued positive definite functions on the group of integers}
We start with one of the simplest cases when the abelian group is $\mathbb{Z}$. All the characterizations in this part shall be given in explicit forms. In addition, a new phenomenon will be demonstrated by certain examples.

Assume $\phi$ is a quaternion-valued positive definite function on $\mathbb{Z}$, as shown above it is spontaneously exponentially bounded. Then Theorem \ref{thm-boundedness} says the representation $\lambda$ of $\mathbb{Z}$ defined as \eqref{eq-unitary-rep} with the semigroup $S$ replaced by the group $\mathbb{Z}$, can be uniquely extended to the RKHS associated with $\phi$. Furthermore, the extension of $\lambda(n)$, $n\in \mathbb{Z}$, is nothing but the shifting operator $L_{n}$ given by
$$L_{n}(f)(m):=f(m-n),$$
for all $f$ in the RKHS and all $m\in \mathbb{Z}$.
In an unpublished work of ours, we have established a spectral characterization of  quaternion-valued positive definite functions on the group of real numbers  based on the S-functional calculus. By the same method as employed in the former work, we can obtain the following result on the group of integers:

There exists a unique pair of a resolution of the identity $E$ on the set $\mathbb{S}^+:=\{s=e^{it}\mid t\in [0,\pi]\}$ and a bounded operator $J$ on the RKHS, subject to the conditions $EJ=JE$ and $J^2=-1$,  such that
$$\lambda(n)=\int_{\mathbb{S}^+}\Re(s^n)dE(s)+J\int_{\mathbb{S}^+}\Im(s^n)dE(s)$$ holds for any integer $n$, where $\Re(s^n)$ and $\Im(s^n)$ stand for the real part and the imaginary part of $s^n$, respectively. This is a quit interesting result in our own opinion. And one can easily see that an integral characterization for the function $\phi$ follows naturally since Equation \eqref{eq-representation} indicates $\phi(g)=\langle \lambda(g)\phi_0,\phi_0\rangle$. But unfortunately, this approach can not solve any of the major questions stated in our introduction on the convex structure of quaternion-valued positive definite functions. So we are forced to find another way to achieve our aims. And we leave the further discussion on the above approach to the possible future work.

We now turn to the fist main question for the group of integers: `how to characterize the extreme boundary of the convex set of normalized quaternion-valued positive definite functions'.

Recall that $\mathbb{S}$ stands for the group of all unit quaternions with the composition being the multiplication inherited from the quaternion algebra, and the convex set of normalized positive definite functions on an abelian group $G$ is denoted by $\mathcal{P}_*^\mathbb{H}(G)$. The theorem below is concerned with the extreme elements in $\mathcal{P}_*^\mathbb{H}(\mathbb{Z})$.

\begin{theorem}\label{Thm-ext-boundary-on-Z}
A function $\phi:\mathbb{Z}\to\mathbb{H}$ is an extreme element in $\mathcal{P}_*^\mathbb{H}(\mathbb{Z})$ if and only if it is a homomorphism from $\mathbb{Z}$ to the sphere group $\mathbb{S}$.
\end{theorem}
More explicitly, every extreme element has the form of $\phi(n)=s^{n}$ with $s\in\mathbb{S}$.

\begin{proof}
First we show the sufficiency. Assume $\phi$ is a homomorphism from $\mathbb{Z}$ to the sphere group $\mathbb{S}$, then $\phi(n-m)=\phi(n)\phi(m)^{-1}=\phi(n)\overline{\phi(m)}$ holds for all integers $n$ and $m$. It follows that the homomorphism $\phi$ is normalized, i.e., satisfying the condition $\phi(0)=1$, and positive definite; which means it belongs to the family $\mathcal{P}_1^\mathbb{H}(\mathbb{Z})$. Moreover, if $\phi$ is a convex combination of two functions $f$ and $g$ in $\mathcal{P}_*^\mathbb{H}(\mathbb{Z})$, namely $\phi=\lambda f + (1-\lambda) g$ with the constant $\lambda\in[0,1]$; then based on the following facts:
\begin{itemize}
  \item[i)] $\lvert\phi(n)\rvert\equiv 1$, since $\phi(n)$ is always a unit quaternion,

  \item[ii)] $\lvert f(n)\rvert, \lvert g(n)\rvert\leq 1$ for every integer $n$, indicated by Theorem \ref{Thm-qpd-to-cpd} together with the boundedness of complex-valued positive definite functions;
\end{itemize}
 it is not difficult to verify that both $f$ and $g$ must be identical with $\phi$. In conclusion, every homomorphism from $\mathbb{Z}$ to $\mathbb{S}$ is an extreme element in $\mathcal{P}_*^\mathbb{H}(\mathbb{Z})$.

Next we show the necessity. Suppose that  $\phi$ is an extreme element in the convex set  $\mathcal{P}_*^\mathbb{H}(\mathbb{Z})$, and for an arbitrary integer $m$ we consider two related functions
$$\tau_m^+ \phi(n):= \phi(n)+\frac{1}{2}\phi(n+m)+\frac{1}{2}\phi(n-m); $$
$$\tau_m^- \phi(n):= \phi(n)-\frac{1}{2}\phi(n+m)-\frac{1}{2}\phi(n-m). $$
We claim that both the functions above are (quaternionic) positive definite on $\mathbb{Z}$. It can be proved in a straightforward way. For any $n_1,n_2,\cdots,n_k\in \mathbb{Z}$  and any $q_1,q_2,\cdots,q_k\in\mathbb{H}$,  by the first conclusion of Theorem \ref{thm-basic-properties}  we have
\begin{eqnarray*}
  \sum_{1\leq i,j\leq k} \overline{q_i}\tau_m^{\pm} \phi(n_j-n_i)q_j
  &=& \sum_{1\leq i,j\leq k} \overline{q_i}\phi(n_j-n_i)q_j \\
  & & \pm\Re \left\{\sum_{1\leq i,j\leq k} \overline{q_i}\phi(n_j-n_i+m)q_j\right\} \\
  &=& \varphi(0)\pm\Re \varphi(m),
\end{eqnarray*}
 where
\begin{equation}\label{eq-def-varphi-proof-integer}
  \varphi(n):=\sum_{1\leq i,j\leq k} \overline{q_i}\phi(n_j-n_i+n)q_j \quad\text{for}\quad n\in\mathbb{Z};
\end{equation}
and the symbol $\Re$ means taking the real part. Theorem \ref{thm-positive-definiteness} says the function $\varphi$ defined in such a form is positive definite, then applying Theorem \ref{Thm-qpd-to-cpd} again in light of the boundedness of complex-valued positive definite functions gives
\begin{equation}\label{Eq-proof-integer-varphi-boundedness}
\lvert \varphi(m)\rvert\leq\varphi(0)\quad\text{for}\quad m\in\mathbb{Z}.
\end{equation}
Accordingly, we get
$$\sum_{1\leq i,j\leq k} \overline{q_i}\tau_m^{\pm} \phi(n_j-n_i)q_j=\varphi(0)\pm\Re \varphi(m)\geq 0,  $$
which confirms our claim. What's more, because $\phi$ is an extreme element in the convex set  $\mathcal{P}_*^\mathbb{H}(\mathbb{Z})$ as supposed and obviously it equals  a convex combination of the positive definite functions $\tau_m^{\pm}\phi$, there must exist a positive constant $C_m$ such that $\tau_m^{+}\phi=C_m\phi$. We thus obtain a vital equality:
\begin{equation}\label{Eq-integer-convex-comb}
  \phi(n+m)+\phi(n-m)=C'_m \phi(n) \quad\text{for}\quad n\in \mathbb{Z},
\end{equation}
with $C'_m$ being a real constant dependent on the parameter $m$.

Based on the above equality, we shall derive that the range of any extreme element $\phi$ in   $\mathcal{P}_1^\mathbb{H}(\mathbb{Z})$ is included by a complex slice. Recall that a complex slice $\mathbb{C}_I$ is a subalgebra of the quaternion algebra generated by an imaginary unit $I$. \eqref{Eq-integer-convex-comb} indicates that if both $\phi(n-m)$ and $\phi(n)$ are contained by the same complex slice, so be $\phi(n+m)$. Besides, $\phi(0)$ is certainly a nonnegative real number, thereby lies in every complex slice. It thus follows by induction on $n$ that all the values $\phi(0), \phi(\pm 1), \cdots, \phi(\pm n)$ must belong to the same complex slice for any positive integer $n$. As a consequence, there exists an imaginary unit $I$ such that the range of $\phi$ is included by $\mathbb{C}_I$.

Then for an arbitrary integer $m$, we define another pair of auxiliary functions $\kappa_m^{\pm}\phi$ by
$$\kappa_m^+ \phi(n):= \phi(n)+\frac{I}{2}\phi(n+m)-\frac{I}{2}\phi(n-m); $$
$$\kappa_m^- \phi(n):= \phi(n)-\frac{I}{2}\phi(n+m)+\frac{I}{2}\phi(n-m). $$
We claim that they are also (quaternionic) positive definite like $\tau_m^{\pm}\phi$. Since $\kappa_m^{\pm}\phi$ are $\mathbb{C}_I$-valued, according to Theorem \ref{thm-equiv-c-h-pd} it suffices to prove that they are (complex) positive definite. The proof of this result is quite similar with that given earlier for the  positive definiteness of the functions $\tau_m^{\pm}\phi$, and so some minor details are omitted.
For any $n_1,n_2,\cdots,n_k\in \mathbb{Z}$  and any $c_1,c_2,\cdots,c_k\in\mathbb{C}_I$, a simple calculation leads to
$$
  \sum_{1\leq i,j\leq k} \overline{c_i}c_j\kappa_m^{\pm} \phi(n_j-n_i)=\varphi^{*}(0)\pm\Re \{I\varphi^{*}(m)\},
$$
where
$$\varphi^{*}(n):=\sum_{1\leq i,j\leq k} \overline{c_i}c_j\phi(n_j-n_i+n) \quad\text{for}\quad n\in\mathbb{Z};$$
It is easy to see that $\varphi^{*}$ is just a special case of the function $\varphi$ given by \eqref{eq-def-varphi-proof-integer} when all the linear combination parameters are limited in the complex slice $\mathbb{C}_I$. Then \eqref{Eq-proof-integer-varphi-boundedness} implies
$$
  \sum_{1\leq i,j\leq k} \overline{c_i}c_j\kappa_m^{\pm} \phi(n_j-n_i)\geq 0,
$$
or equivalently, $\kappa_m^{\pm}\phi$ are (complex) positive definite. Thus our second claim is also confirmed.
Moreover, both the functions $\kappa_m^{\pm}\phi$ must be nonnegative scalar multiples of $\phi$, since $\phi$ is an extreme element in $\mathcal{P}_*^\mathbb{H}(\mathbb{Z})$, namely the set of normalized positive definite functions, and equals a convex combination of the functions $\kappa_m^{\pm}\phi$ which are also positive definite (but possibly not normalized). It follows immediately that there exists a real constant $C''_m$ depending on the parameter $m$ such that
\begin{equation*}
  \phi(n+m)-\phi(n-m)=IC''_m \phi(n) \quad\text{for}\quad n\in \mathbb{Z}.
\end{equation*}
Combining this equality and its analogue \eqref{Eq-integer-convex-comb}, we see that
\begin{equation*}
  \phi(n+m)=c(m) \phi(n) \quad\text{for}\quad n,m\in \mathbb{Z},
\end{equation*}
where $c(m)$ is a certain $\mathbb{C}_I$-valued function of the variable $m$. Setting $n=0$, we get $c(m)=\phi(m)$. Consequently, \begin{equation*}
  \phi(m+n)= \phi(m) \phi(n)\quad\text{for}\quad n,m\in \mathbb{Z},
\end{equation*}
In addition, Theorem \ref{thm-basic-properties} says every positive definite function is hermitian, thereby $\phi(-n)=\overline{\phi(n)}$. Then it is easy to see that
$$\lvert\phi(n)\rvert^2=\phi(-n)\phi(n)=\phi(0)=1\quad\text{for}\quad n\in \mathbb{Z}. $$
Therefore, $\phi$ is a homomorphism from $\mathbb{Z}$ to the sphere group $\mathbb{S}$. The proof is now completed.
\end{proof}

\begin{remark}
  The key to the proof of Theorem \ref{Thm-ext-boundary-on-Z} is to show that the range of every extreme element in
  the convex set of normalized positive definite functions on the group of integers is included by some complex slice $\mathbb{C}_I$. This result ensures the positive-definiteness of the second pair of auxiliary functions $\kappa_m^{\pm}\phi$.
\end{remark}

Theorem \ref{Thm-ext-boundary-on-Z} yields a one-to-one correspondence between the sphere group $\mathbb{S}$ and the extreme boundary of $\mathcal{P}_*^\mathbb{H}(\mathbb{Z})$, the expression of which is given as $$s\in\mathbb{S} \longmapsto \phi\in Ex \left(\mathcal{P}_*^\mathbb{H}(\mathbb{Z})\right):n\mapsto s^n. $$
Apparently this bijection is continuous with respect to the pointwise convergence topology on the extreme boundary $Ex \left(\mathcal{P}_*^\mathbb{H}(\mathbb{Z})\right)$, and thereby is a homeomorphism since its definition domain is compact and its range is Hausdorff, which means that it is legitimate  in the topological sense to identify the sphere group with the extreme boundary. Then a integral characterization of quaternionic positive definite functions on the group of integers follows immediately.

\begin{theorem}\label{Thm-characterization-Z}
  A function $\phi:\mathbb{Z}\mapsto\mathbb{H}$ is positive definite if and only if there exists a (not necessarily unique) nonnegative Radon measure $\mu$  on the sphere group $\mathbb{S}$ such that
\begin{equation}\label{Eq-thm-characterization-on-Z-integral-expression}
  \phi(n)=\int_{\mathbb{S}}s^n d\mu(s), \quad\text{for}\quad n\in \mathbb{Z}.
\end{equation}
\end{theorem}
\begin{proof}
The sufficiency can be easily verified based on the fact that  for any $s\in\mathbb{S}$, $s^n$ is positive definite with respect to the integer variable $n$.

As for the necessity, it suffices to prove that every  $\phi\in\mathcal{P}_*^\mathbb{H}(\mathbb{Z})$ can be expressed in the form of \eqref{Eq-thm-characterization-on-Z-integral-expression} with $\mu$ being a regular Borel probability measure. Recall that on the group of integers, every positive definite function is bounded, in other words dominated by the constant absolute value $1$. Hence, according to Corollary \ref{cor-compactness} with the absolute value $\alpha\equiv 1$, the convex set $\mathcal{P}_*^\mathbb{H}(\mathbb{Z})$ is compact. Then in light of Theorem 3.23 (the Krein-Milman Theorem) and Theorem 3.28 in \cite{Rudin-1991}, or the Choquet-Bishop-de Leeuw Theorem, we see that for any $\phi\in\mathcal{P}_*^\mathbb{H}(\mathbb{Z})$, there is a regular Borel probability measure
$\mu$ on the extreme boundary of $\mathcal{P}_*^\mathbb{H}(\mathbb{Z})$ such that the following equality
\begin{equation*}
  \phi=\int_{Ex \left(\mathcal{P}_*^\mathbb{H}(\mathbb{Z})\right)}\varphi \quad d\mu(\varphi).
\end{equation*}
holds in the sense of vector-valued integration. And noting the evident fact that taking the value at any integer point is a continuous operation with respect to the pointwise convergence topology, we thus have
\begin{equation*}
  \phi(n)=\int_{Ex \left(\mathcal{P}_*^\mathbb{H}(\mathbb{Z})\right)}\varphi(n) \quad d\mu(\varphi), \quad\text{for}\quad n\in \mathbb{Z}.
\end{equation*}
As shown in the argument right above this theorem, the sphere group $\mathbb{S}$ can be identified with the extreme boundary $Ex \left(\mathcal{P}_*^\mathbb{H}(\mathbb{Z})\right)$ via the following homeomorphism
$$s\in\mathbb{S} \longmapsto \phi\in Ex \left(\mathcal{P}_*^\mathbb{H}(\mathbb{Z})\right):n\mapsto s^n. $$
This observation allows us to rewrite the preceding equality as  \eqref{Eq-thm-characterization-on-Z-integral-expression} equivalently as desired. The proof is now completed.
\end{proof}

In the remark below, an answer to the second main question stated in the introduction, namely `whether the convex set of normalized quaternion-valued positive definite functions is a Bauer simplex in the pointwise convergence
topology' is given in the case of the group of integers; and two explicit examples are presented to support our assertion.

\begin{remark}\label{Re-Z}
  In terms of convexity the biggest difference between the complex case and the quaternionic one is how the extreme elements in the family of normalized positive definite functions are distributed. More specifically, the set of complex-valued normalized positive definite functions on any abelian group is a Bauer simplex, that is to say, every real valued continuous function on its extreme boundary has a unique affine extension to the whole set (see Theorem 4.3 in Chapter II of \cite{Alfsen-1971}); in contrast, its quaternionic counterpart is not. Intuitively speaking, in the quaternionic setting the distribution of the extreme elements are more uniform.

  This major difference also can be observed from the viewpoint of integral characterization. For the sake of simplicity and content coherence, we limit our argument in the case of the group of integers instead of a random abelian group at the present stage. Herglotz representation theorem indicates that for a complex-valued normalized positive definite function $\phi$, there exists a unique regular Borel probability measure $\mu$ on the unit circle $\mathbb{T}$ in the complex plane such that
$$\phi(n)=\int_{\mathbb{T}}z^n d\mu(z) \quad\text{for}\quad n\in \mathbb{Z}. $$
  In a word, the existence and uniqueness of the measure representation is valid. Then a question arises: whether this still holds true in the quaternionic setting. The answer turns out partially negative. The existence is valid as stated in Theorem \ref{Thm-characterization-Z}, but the uniqueness is actually not as illustrated by the two examples below.
 \begin{itemize}
  \item[Ex 1:] Consider $\phi(n)=\alpha^n$ with the constant $\alpha$ being a unit quaternion, by Theorem \ref{Thm-ext-boundary-on-Z}, we know that it is an extreme element in the family of normalized positive definite functions. Thus only one nonnegative measure $\mu=\delta_\alpha$, namely the Dirac measure centred on the point $\alpha$, satisfies \eqref{Eq-thm-characterization-on-Z-integral-expression}. In other words, $\phi$ can not be represented as the barycenter of any nonnegative measure except $\delta_\alpha$.

  \item[Ex 2:] Observe the cosine function $\phi(n)=\cos n$, one can easily see $$\phi(n)=\frac{1}{2}(\alpha^n+\overline{\alpha}^n)\quad\text{for}\quad n\in \mathbb{Z},\ \alpha\in\mathbb{S}.$$
      Thus the measure $\mu=\frac{1}{2}(\delta_\alpha+\delta_{\overline{\alpha}})$ satisfies \eqref{Eq-thm-characterization-on-Z-integral-expression} for every $\alpha=e^I$ with $I$ being a pure imaginary unit quaternion, which means the uniqueness of the measure representation breaks down in this example.
 \end{itemize}
 In conclusion, only for some of the quaternion-valued positive definite functions, the corresponding measure representations are unique. And further arguments on this phenomenon in the general case of an arbitrary abelian group are to come at the end of this section where a criterion for which of the measure representations, i.e., nonnegative Radon measure solutions to \eqref{Eq-thm-characterization-on-Z-integral-expression}, are unique shall be established.
\end{remark}

\subsection{generalized results for quaternion-valued positive definite functions on arbitrary abelian groups}In this part, we focus on the former of the two major questions, i.e., `how to characterize the extreme boundary of the convex set of normalized quaternion-valued positive definite functions' in the general case of arbitrary abelian groups.
The trick of our approach is still to show that the range of every extreme element is included by a complex slice as in the case of the group of integers. By contrast, a more complicated manipulation is required in the general case.

For simplicity of presentation, hereinafter we always assume that $G$ is an abelian group equipped with the inverse involution. In preparation for establishing our first main result in this case, we need the following two lemmas.
\begin{lemma}\label{Lemma-Ext-1}
  Any extreme element $\phi$ in $\mathcal{P}_*^\mathbb{H}(G)$ satisfies the equations below.
 \begin{align*}
   2\Re \phi(b)\Re \phi(a) &=\Re \phi(a+b)+\Re \phi(a+b^*), \\
   2\Re \phi(b)\Im \phi(a) &=\Im \phi(a+b)+\Im \phi(a+b^*),
 \end{align*}
 for all $a, b\in G$.
\end{lemma}
Here operations $\Re$ and $\Im$ take the real part and the imaginary part of any given quaternion respectively. More explicitly, for $q=a_0i_0+a_1i_1+a_2i_2+a_3i_3\in\mathbb{H}$,
$$\Re\ q=a_0i_0 \quad \text{and} \quad \Im\ q=a_1i_1+a_2i_2+a_3i_3.$$
\begin{proof}
Suppose that  $\phi$ is an extreme element in the convex set  $\mathcal{P}_*^\mathbb{H}(\mathbb{G})$, and for an arbitrary element $b\in G$ we consider two related functions
$$\tau_b^\pm \phi(a):= \phi(a)\pm\frac{1}{2}\phi(a+b)\pm\frac{1}{2}\phi(a+b^*)  \quad\text{for}\quad a\in G. $$
Adopting the same procedure as in the proof of Theorem \ref{Thm-ext-boundary-on-Z}, we can easily deduce that  $\tau_b^\pm \phi$ are both positive definite on $G$. Then, since
$\phi$ is an extreme element in the convex set  $\mathcal{P}_*^\mathbb{H}(\mathbb{G})$ and equals a convex combination of $\tau_b^{\pm}\phi$, there exists a positive constant $C_b$ dependent on $b$ such that $\tau_b^{+}\phi=C_b\phi$. Accordingly, we have
\begin{equation}\label{eq-proof-Lemma-Ext-1}
  \phi(a+b)+\phi(a+b^*)=(2C_b-1) \phi(a) \quad\text{for}\quad a\in G.
\end{equation}
Moreover, by putting $a=0$, we obtain $2C_b-1=\phi(b)+\phi(b^*)$. In addition, Theorem \ref{thm-basic-properties} says $\phi$ is hermitian. Thus the equality
\begin{equation*}
  2C_b-1=2\Re\phi(b).
\end{equation*}
holds true. Inserting this back into \eqref{eq-proof-Lemma-Ext-1} and separating the real and imaginary parts of the both sides finally leads to the desired result.
\end{proof}

We would like to mention that the results of Lemma \ref{Lemma-Ext-1} is also valid in the case of an abelian semigroup with involution. And here comes the other technical lemma.

\begin{lemma}\label{Lemma-Ext-2}
 Let $\phi$ be a quaternion-valued solution to the following system
  \begin{align}
   \label{eq-Lemma-Ext-2-1}2\Re \phi(b)\Re \phi(a) &=\Re \phi(a+b)+\Re \phi(a+b^*), \\
   \label{eq-Lemma-Ext-2-2}2\Re \phi(b)\Im \phi(a) &=\Im \phi(a+b)+\Im \phi(a+b^*),
 \end{align}

 for all $a, b\in G$, subject to the given condition that $\phi(0)=1$ and $\displaystyle{\sup_G\lvert\phi\rvert=1}$. Then there exists an imaginary unit $I\in\mathbb{H}$ such that the range of $\phi$ is included by the complex slice $\mathbb{C}_I$.
\end{lemma}
\begin{proof}
We divide our proof into three steps. In each step, one of the following properties shall be verified in sequence.
\begin{itemize}
  \item[i)] For every $a\in G$, there is a complex slice which contains $\phi(na)$ for all integers $n$.
  \item[ii)] For any $a, b\in G$, there is one which contains $\phi(n(a+b))$ and $\phi(n(a+b^*))$ for all integers $n$.
  \item[iii)] For any $a, b\in G$, there is one which contains both $\phi(a)$ and $\phi(b)$.
\end{itemize}
The details are as follows.
\bigskip

Step 1:

Assume that $a$ is an arbitrary element in $G$. We define
 $$f(n):=\phi(na)\quad\text{for}\quad n\in\mathbb{Z}. $$
Then we substitute this expression into \eqref{eq-Lemma-Ext-2-2} to obtain
 \begin{equation}\label{eq-Lemma-Ext-2-3}
    2\Re f(m)\Im f(n) =\Im f(n+m)+\Im f(n-m)\quad\text{for}\quad n,m\in\mathbb{Z},
 \end{equation}
which implies if a complex slice contains both $f(n)$ and $f(n-m)$, it will also contains $f(n+m)$.
Accordingly, an obvious induction on $n$ shows that all the values
$$f(0), f(\pm 1),\cdots, f(\pm n)$$
are in the same complex slice. Hence, there exists an imaginary unit $I_a$ such that the complex slice $C_{I_a}$ contains $f(n)$ for all integers $n$. In other words, $f$ is essentially complex-valued.

In addition, substituting the expression of $f$ into \eqref{eq-Lemma-Ext-2-2} gives
 \begin{equation*}
    2\Re f(m)\Re f(n) =\Re f(n+m)+\Re f(n-m)\quad\text{for}\quad n,m\in\mathbb{Z}.
 \end{equation*}
One may notice that the equality above and \eqref{eq-Lemma-Ext-2-3} share the same form with the product-to-sum formulae of trigonometric functions. It is not difficult to see from this viewpoint that
there exist two constants $\theta_a\in[0,2\pi)$ and $t_a\in[0,1]$ so that
\begin{equation}\label{eq-Lemma-Ext-2-4-key}
  \phi(na)=f(n)=\cos (n\theta_a)+t_a\sin(n\theta_a)I_a\quad\text{for}\quad n\in\mathbb{Z}.
\end{equation}
The first part of this proof is now completed.

\bigskip
Step 2:

In this part, we attempt to prove by contradiction that for any given $a, b\in G$, all the values $\phi(n(a+b))$ and $\phi(n(a+b^*))$ with $n\in\mathbb{Z}$ lie in one complex slice. First, we assume the opposite of the conclusion:  for some $a, b\in G$, these related values can not be put in the same complex slice.

The next thing to do is to derive new consequences until one violates the premise.
The final conclusion of the first step states that there are imaginary units $I_a,I_b\in\mathbb{H}$, angles $\theta_a,\theta_b\in[0,2\pi)$ and real constants $t_a, t_b\in[0,1]$ such that the following expressions
 \begin{align*}
   \phi(na)&=\cos (n\theta_a)+t_a\sin(n\theta_a)I_a,\\
   \phi(nb)&=\cos (n\theta_b)+t_b\sin(n\theta_b)I_b,
 \end{align*}
hold for all integers $n$. Similarly, we have
 \begin{align*}
   \phi(n(a+b))&=\cos (n\theta_1)+t_1\sin(n\theta_1)I_1, \\
   \phi(n(a+b^*))&=\cos (n\theta_2)+t_2\sin(n\theta_2)I_2.
 \end{align*}
Moreover, the premise above indicates that the imaginary units $I_1,I_2$ are real linearly independent, the angles $\theta_1,\theta_2$ belong to $(0,\pi)\cup(\pi,2\pi)$, and the constants $t_1, t_2$ lie in $(0,1]$. Multiplying $a,b$ by $n$ in \eqref{eq-Lemma-Ext-2-2} and inserting the four preceding expressions leads to the equality below.
\begin{equation*}
2\cos(n\theta_b)\sin(n\theta_a)I_a=t_1\sin(n\theta_1)I_1+t_2\sin(n\theta_2)I_2\quad\text{for}\quad n\in\mathbb{Z}.
\end{equation*}
When $n=1$, the right side is clearly not zero, which forces the imaginary unit $I_a$ on the left side to become a real linear combination of the ones on the right side. Hence, there exist two real constants $c_i$ such that
\begin{equation}\label{eq-Lemma-Ext-2-5}
  \sin(n\theta_i)=c_i\cos(n\theta_b)\sin(n\theta_a), \quad i=1,2;
\end{equation}
are valid for all integers $n$.
By performing the same procedure with interchange of $a$ and $b$, we have
\begin{equation}\label{eq-Lemma-Ext-2-6}
  \sin(n\theta_i)=c'_i\cos(n\theta_a)\sin(n\theta_b), \quad i=1,2;
\end{equation}
where $c'_i$ is another pair of real constants. Thus,
\begin{equation*}
  \sin(n\theta_i)=\frac{1}{2}(c_i+c'_i)\sin(n(\theta_a+\theta_b))=\frac{1}{2}(c_i-c'_i)\sin(n(\theta_a-\theta_b)), \quad i=1,2;
\end{equation*}
holds for all $n$, implying  $\theta_a\pm\theta_b\equiv\theta_i$ or $-\theta_i$ mod $2\pi$ in the conventional additive group structure of $\mathbb{R}$. It follows that $\theta_a$ or $\theta_b\equiv0$ mod $\pi$, so in view of \eqref{eq-Lemma-Ext-2-5} and \eqref{eq-Lemma-Ext-2-6},
\begin{equation*}
\sin(n\theta_1)=\sin(n\theta_2)=0\quad\text{for}\quad n\in\mathbb{Z},
\end{equation*}
which is in contradiction to the condition $\theta_1,\theta_2\in(0,\pi)\cup(\pi,2\pi)$.

Therefore, the premise is false. Then indeed, for any given $a, b\in G$, there is an imaginary unit $I$ such that the complex slice $C_I$ contains all $\phi(n(a+b))$ and $\phi(n(a+b^*))$. Now the second part of this proof is accomplished.

\bigskip
Step 3:

It only remains to show that for any $a, b\in G$, there is a complex slice which contains both $\phi(a)$ and $\phi(b)$. The approach is similar with that applied in the second step, and some minor details are omitted to avoid repetition.

Let $a, b\in G$, and suppose that no complex slice can contain both $\phi(a)$ and $\phi(b)$. In order to create a contradiction, now we need to gather some preliminary results. Firstly, owing to the final conclusion of the first step, the values of $\phi$ on the two subgroups generated by $a$ and $b$ respectively have the forms as follows.
 \begin{align*}
   \phi(na)&=\cos (n\theta_a)+t_a\sin(n\theta_b)I_a, \\
   \phi(nb)&=\cos (n\theta_b)+t_b\sin(n\theta_b)I_b.
 \end{align*}
with $I_a,I_b$ being a pair of real linearly independent imaginary units, $\theta_a,\theta_b\in(0,\pi)\cup(\pi,2\pi)$ and $t_a, t_b\in (0,1]$. Secondly, according to the final conclusion in the second step, all $\phi(n(2b+a^*))$ and $\phi(na)$ with $n\in\mathbb{Z}$ are in the same complex slice, since $2b+a^*=b+(a+b^*)^*$ and $a=b+(a+b^*)$. This means $\phi(n(2b+a^*))$ must belong to $C_{I_a}$ because this complex slice is the only one that contains all $\phi(na)$. Then applying the final result of the first step again gives
\begin{equation*}
\phi(n(2b+a^*))=\cos (n\theta_3)+t_3\sin(n\theta_3)I_a,
\end{equation*}
where $\theta_3\in[0,2\pi)$ and $t_3\in [0,1]$; while for the values of $\phi$ on the subgroup generated by $a+b^*$, we have
\begin{equation*}
\phi(n(a+b^*))=\cos (n\theta_4)+t_4\sin(n\theta_4)I_4,
\end{equation*}
where $\theta_4\in[0,2\pi)$ and $t_4\in [0,1]$ and $I_4$ is another imaginary unit, the relation between which and $I_a,I_b$ is unknown yet.

Next, we precede to create a contradiction. Replacing $a$ and $b$ with $b$ and $a+b^*$ respectively in \eqref{eq-Lemma-Ext-2-1} and \eqref{eq-Lemma-Ext-2-2} and then substituting the four expressions about $\phi$ above yields
 \begin{align}
  \label{eq-Lemma-Ext-2-7}\cos (n\theta_4)\cos (n\theta_b)&=\cos (n\theta_a)+\cos (n\theta_3), \\
  \label{eq-Lemma-Ext-2-8}\cos (n\theta_4)\sin (n\theta_b)I_b&=\left(t_a\sin (n\theta_a)+t_3\sin (n\theta_3)\right)I_a.
 \end{align}
It is immediately seen from \eqref{eq-Lemma-Ext-2-8} that the equalities below
 \begin{align}
  \label{eq-Lemma-Ext-2-9}\cos (n\theta_4)\sin (n\theta_b)&=0, \\
  \label{eq-Lemma-Ext-2-10}t_a\sin (n\theta_a)+t_3\sin (n\theta_3)&=0.
 \end{align}
hold for all integers $n$, since $I_a$ and $I_b$ are real linearly independent. The condition $\theta_b\in(0,\pi)\cup(\pi,2\pi)$, in light of \eqref{eq-Lemma-Ext-2-9}, leads to the result that  $\theta_b$ and $\theta_4\equiv \frac{\pi}{2}$ mod $\pi$. Hence, the left side of \eqref{eq-Lemma-Ext-2-7} equals $0$ when $n$ is odd, and $2$ when $n$ is even.
What's more, applying the condition that $\theta_a\in(0,\pi)\cup(\pi,2\pi)$ and $t_a\neq 0$ to \eqref{eq-Lemma-Ext-2-10}, we get
$t_a=t_3$ and $\theta_a=2\pi-\theta_3$, so that the right side of \eqref{eq-Lemma-Ext-2-7} equals
$2\cos(n\theta_a)$. Therefore,
\begin{equation*}
  \cos(n\theta_a)=\left\{\begin{array}{lcl}
                           0, &  & \text{for } n \text{ odd}; \\
                           &  &\\
                           1, &  & \text{for } n \text{ even}.
                         \end{array}\right.
\end{equation*}
It is clear that no such angle $\theta_a$ exists. Finally, the opposite of the desired final conclusion is reduced to this absurdity, which completes the whole proof.
\end{proof}

We can now give an answer to the first main question. Recall that the notation $\mathbb{S}$ stands for the $3$-dimensional sphere consisting of unit quaternions, and $\mathcal{P}_*^\mathbb{H}(G)$ the family of normalized quaternion-valued positive definite functions on $G$.

\begin{theorem}\label{Thm-ext-boundary-on-G}
A function $\phi:G\to\mathbb{H}$ is an extreme element in $\mathcal{P}_*^\mathbb{H}(G)$ if and only if it is a homomorphism from $G$ to the sphere group $\mathbb{S}$.
\end{theorem}

\begin{proof}
First, we verify the necessity. Provided that $\phi$ is an extreme element in the convex set  $\mathcal{P}_*^\mathbb{H}(\mathbb{Z})$, combining Lemmas \ref{Lemma-Ext-1} and \ref{Lemma-Ext-2} yields that there exists a imaginary unit $I\in\mathbb{H}$ such that the corresponding complex slice $C_I$ includes the range of $\phi$.
Then for an arbitrary $b\in G$, we define a pair of auxiliary functions $\kappa_b^{\pm}\phi$ as
$$\kappa_b^+ \phi(a):= \phi(a)+\frac{I}{2}\phi(a+b)-\frac{I}{2}\phi(a+b^*); $$
$$\kappa_b^- \phi(a):= \phi(a)-\frac{I}{2}\phi(a+b)+\frac{I}{2}\phi(a+b^*). $$
An argument similar to the one used in the proof of Theorem \ref{Thm-ext-boundary-on-Z} shows that
$\kappa_b^{\pm}\phi$, treated as $C_I$-valued functions, both are complex positive definite, thereby quaternionic positive definite due to Theorem \ref{thm-equiv-c-h-pd}.

Obviously, the extreme element $\phi$ is a convex combination of $\kappa_b^{\pm}\phi$, and then the quaternionic positive definiteness of both the functions $\kappa_b^{\pm}\phi$ forces them to be nonnegative scalar multiples of $\phi$. We thus obtain
\begin{equation*}
  \phi(a+b)-\phi(a+b^*)=IC_b\phi(a)\quad\text{for}\quad a\in G,
\end{equation*}
where $C_b$ is a real constant dependent on $b$. Then putting $a=0$, we get $IC_b=\phi(b)-\phi(b^*)=2\Im \phi(b)$ by Theorem \ref{thm-basic-properties}.
Hence,
\begin{equation*}
  \phi(a+b)-\phi(a+b^*)=2\Im \phi(b)\phi(a)\quad\text{for}\quad a,b\in G.
\end{equation*}
Moreover, Lemma \ref{Lemma-Ext-1} says
\begin{equation*}
  \phi(a+b)+\phi(a+b^*)=2\Re \phi(b)\phi(a)\quad\text{for}\quad a,b\in G.
\end{equation*}
Therefore,
\begin{equation*}
  \phi(a+b)=\phi(b)\phi(a)\quad\text{for}\quad a,b\in G;
\end{equation*}
which, in light of Theorem \ref{thm-basic-properties}, implies
$\lvert \phi(a)\rvert^2=\phi(a^*)\phi(a)=\phi(0)=1$. In conclusion, every extreme element $\phi$ is a homomorphism from $G$ to the sphere group $\mathbb{S}$.

As for the sufficiency, its proof is just a replication of that given earlier for the sufficiency in Theorem \ref{Thm-ext-boundary-on-Z}, and so is omitted.
\end{proof}

As shown in the above theorem, homomorphisms from $G$ to the sphere group $\mathbb{S}$ play a significant role in the convex structure of quaternion-valued positive definite functions. And they shall be named of a new term.

\begin{definition}
  A quaternionic character on $G$ is a homomorphism from $G$ to the sphere group $\mathbb{S}$. The set of all the quaternionic characters is called the quaternionic dual of $G$ and denoted by $G^\delta$, namely $G^\delta=Hom(G,\mathbb{S})$.
\end{definition}
We refer to the pointwise convergence topology as the canonical topology of $G^\delta$.
With the new notation, Theorem \ref{Thm-ext-boundary-on-G} can be simplified as $$Ex\left(\mathcal{P}_*^\mathbb{H}(G)\right)=G^\delta. $$
Note that in contrast to its complex counterpart, i.e., $Hom(G,\mathbb{T})$ usually called the Pontryagin dual or the dual group of $G$,
the quaternionic dual $G^\delta$ does not possess a natural group structure because of the non-commutativity of quaternions. Nevertheless, there still are many operations on $G^\delta$ with algebraic effects, and some of them will be introduced in the next subsection.

As an application of our first main result, we now present an integral characterization for quaternion-valued positive definite functions.
\begin{theorem}\label{Thm-characterization-G}
  A function $\phi:G\mapsto\mathbb{H}$ is positive definite if and only if there exists a (not necessarily unique) nonnegative Radon measure $\mu$  on $G^\delta$ such that
\begin{equation}\label{Eq-thm-characterization-on-G-integral-expression}
  \phi(g)=\int_{G^\delta}\gamma(g) d\mu(\gamma) \quad\text{for}\quad g\in G.
\end{equation}
\end{theorem}
\begin{proof}
The proof of this result is almost the same with that given for Theorem \ref{Thm-characterization-Z}, and so is omitted.
\end{proof}

\subsection{final arguments on the uniqueness of measure representations on the quaternionic dual}
We now turn to the second main question, which is `whether $\mathcal{P}_*^\mathbb{H}(G)$ is a Bauer simplex in the pointwise convergence topology'. As is well known, there are multiple ways to describe a Bauer simplex (see, e.g., Theorems 4.1 and 4.3 in Chapter II of \cite{Alfsen-1971}). Among them we choose the one directly related to the measure representation given in Theorem \ref{Thm-characterization-G}. Indeed, if the nonnegative solution $\mu$ of  \eqref{Eq-thm-characterization-on-G-integral-expression} is unique for every $\phi\in \mathcal{P}_*^\mathbb{H}(G)$, then it is a Bauer simplex, and vice versa.

In the special case of $G$ being the group of integers, as Remark \ref{Re-Z} says, the nonnegative solution is unique for some $\phi$'s, whereas the uniqueness breaks down for others; which means $\mathcal{P}_*^\mathbb{H}(G)$ is not a Bauer simplex when $G=\mathbb{Z}$.

A natural question arises at this point: Does this phenomenon exist in any case? We shall see later that the answer varies with a certain algebraic feature of $G$. Two approaches to the final conclusion are provided as follows.
\begin{itemize}
  \item[1)] Lemma \ref{lemma-criteria-G} $\longrightarrow$ Theorem \ref{thm-property-p-d-f-G} (with Lemma \ref{lemma-equiv-algebra-feature}) $\longrightarrow$ Theorem \ref{thm-not-Bauer-G}

  \item[2)] A construction method (with Lemma \ref{lemma-equiv-algebra-feature}) $\longrightarrow$ Theorem \ref{thm-not-Bauer-G}
\end{itemize}
The former involves some manipulations on measures and is relatively abstract, while the latter is much more straightforward.

Throughout this subsection, the space of (signed) Radon measures on the quaternionic dual $G^\delta$ is denoted by $\mathcal{M}(G^\delta)$, and the subspace $\mathcal{K}(G^\delta)$ consists of $\nu\in
\mathcal{M}(G^\delta)$ satisfying
$$
\int_{G^\delta}\gamma\ d\nu(\gamma)=0
$$

 In the lemma below, we give a criteria for whether the nonnegative solution to \eqref{Eq-thm-characterization-on-G-integral-expression} is unique.
\begin{lemma}\label{lemma-criteria-G}
Let $\mu\in\mathcal{M}(G^\delta)$ be a nonnegative solution of \eqref{Eq-thm-characterization-on-G-integral-expression} for a given $\phi\in\mathcal{P}_*^\mathbb{H}(G)$. Then it is the unique one iff every nonzero $\nu\in\mathcal{K}(G^\delta)$ satisfies one of the following conditions.
\begin{itemize}
  \item[i)]$\nu_2^-\neq 0$,
  \item[ii)]$\nu_2^-=0$ and $d\nu_1^-/d\mu$ is essentially unbounded,
\end{itemize}
where $\nu^-=\nu_1^-+\nu_2^-$ is the Lebesgue decomposition of the lower variation of $\nu$
into the absolutely continuous part $\nu_1^-$ and the singular part $\nu_2^-$ w.r.t. $\mu$.
\end{lemma}

\begin{proof}
Assume that \eqref{Eq-thm-characterization-on-G-integral-expression} has multiple nonnegative solutions, then there exists another solution $\mu'$ except for $\mu$. One can easily see that nonzero Radon measure $\nu:=\mu'-\mu$ does not satisfy any of Conditions i and ii. Indeed, the lower variation of $\nu$ is absolutely continuous with respect to $\mu$, and furthermore the corresponding Radon-Nikodym derivative is essentially bounded. So the sufficiency is true.

As for the necessity, provided that a nonzero Radon measure $\nu$ fails to satisfies either of Conditions i and ii, which means the Radon-Nikodym derivative $d\nu^-/d \mu$  exists and is essentially less than a positive constant $C$, then we find another different solution $\mu'$ defined as $\mu+\frac{1}{C}\nu$. Now this proof is completed.
\end{proof}

For convenience, we identify the circle group $\mathbb{T}$ with the unit circle in a fixed complex slice $C_{\widehat I}$ of the quaternion algebra. Then the classical dual group $\widehat{G}:=Hom(G,\mathbb{T})$ can be considered as a subset of the quaternion dual $G^\delta$. And we denote the subset of all the real-valued elements (e.g., the constant character $1$)  in $G^\delta$ by $G^\delta_\mathbb{R}$. In particular, $G^\delta_\mathbb{R}\subset \widehat{G} \subset G^\delta$.

The following lemma shows that the above inclusion relation is tight if and only if $G$ has an exponent $\leq 2$. Although its proof is evident, this result is crucial in both the approaches to Theorem \ref{thm-not-Bauer-G}.
\begin{lemma}\label{lemma-equiv-algebra-feature}
The following conditions are equivalent.
\begin{itemize}
  \item[i)]$\widehat{G}=G^\delta$,
  \item[ii)]every $\phi\in\mathcal{P}_*^\mathbb{H}(G)$ is real valued (i.e.,  $G^\delta_\mathbb{R}=G^\delta$),
  \item[iii)] $G$ is of exponent $\leq 2$.
\end{itemize}
\end{lemma}

\begin{proof}
(i $\Rightarrow$ iii) Suppose that $G$ has an exponent $>2$, then there exists an element $a\in G$ satisfying  $a\neq -a$. This allows us to define a function on $G$ as
$$\phi(x):=\left\{\begin{array}{ll}
                   1,& \quad x=0, \\
                   \pm \frac{J}{2},&\quad x=\pm a,  \\
                   0,& \quad otherwise;
                 \end{array}\right. $$
where $J$ is a unit quaternion $\neq \pm\widehat I$. A simple calculation tells us that $\phi$ is positive definite. It is apparent that the range of $\phi$ is not included by $C_I$, which indicates $\widehat{G}\neq G^\delta$ in view of Theorem \ref{Thm-characterization-G}.

(iii $\Rightarrow$ ii) Let $G$ be an abelian group of exponent $\leq 2$. Then $x=-x$ holds for all $x\in G$. We thus have $\phi(x)=\overline{\phi(-x)}=\overline{\phi(x)}$ for $\phi\in\mathcal{P}_*^\mathbb{H}(G)$, since $\phi$ is hermitian according to Theorem \ref{thm-basic-properties}.

(ii $\Rightarrow$ i) Clearly, this part is trivial.
\end{proof}

Although the quaternionic dual $G^\delta$ has no natural group structure, we still can find many operations on $G^\delta$ with algebraic meaning, for example, the involution $\phi^*:=\overline{\phi}$
and the rotation $R_q\phi:= q\phi q^{-1}$ with $q$ being a unit quaternion. It can be easily seen that both the operations preserve the pointwise convergence topology. They thus induce two push-forward mappings $\mu\mapsto \mu^{\hat *}$ and $\mu\mapsto \hat R_q\mu$ on $\mathcal{M}(G^\delta)$ as follows.
 $$\mu^{\hat *}(\Omega):=\mu(\Omega^*)\quad\text{and}\quad \hat R_q\mu(\Omega):=\mu( R_q^{-1}\Omega),\quad\forall\text{ Borel set } \Omega\subset G^\delta. $$
Note that $\Omega^*$ stands for the image and also the inverse image of $\Omega$ through the involution, and $R_q^{-1}\Omega$ does for the inverse image through the rotation.

Now we present an interesting property of the unique measure representations.
\begin{theorem}\label{thm-property-p-d-f-G}
If $\mu\in\mathcal{M}(G^\delta)$ is the unique nonnegative solution of \eqref{Eq-thm-characterization-on-G-integral-expression} for some given $\phi\in\mathcal{P}_*^\mathbb{H}(G)$, then the restrictions of $\mu^{\hat *}$ and $\mu$ on $G^\delta\setminus G^\delta_\mathbb{R}$ are mutually singular, or in short $\mu^{\hat *}\perp\mu$ on $G^\delta\setminus G^\delta_\mathbb{R}$.
\end{theorem}
\begin{proof}
Assume that the restrictions of $\mu^*$ and $\mu$ on $G^\delta\setminus G^\delta_\mathbb{R}$ are not mutually singular.
Due to Lemma \ref{lemma-criteria-G}, it suffices to show that there exists a (signed) measure $\nu\in \mathcal{K}(G^\delta)$ such that its lower variation $\nu^-$ is absolutely continuous with respect to $\mu$ and the Radon-Nikodym derivative $d \nu^- /d \mu$ is essentially bounded.

In order to find such measure $\nu$, we need to construct a neighbour with a well-chosen separation property for every element $\phi\in G^\delta\setminus G^\delta_\mathbb{R}$. The details is as follows.
Since $\phi$ is not real-valued, we have $\phi^*\neq \phi$. Moreover, as demonstrated in the proof of Theorem \ref{Thm-ext-boundary-on-G}, there exists an imaginary unit $I$ such that the complex slice $C_{I}$ includes the range of $\phi$, which is identical with that of $\phi^*$. Take $J$ as another  imaginary unit that anti-commutes with  $I$ and define a quaternion as $q:=\frac{1}{\sqrt 2}(1+J)$. Then we get four distinct elements in $G^\delta\setminus G^\delta_\mathbb{R}$, namely $\phi$, $\phi^*$, $R_q\phi$ and $R_q\phi^*$. And there are neighbours of each which are disjoint from each other, since  $G^\delta\setminus G^\delta_\mathbb{R}$ is Hausdorff. We denote these neighbours by $U_i$ ($i=1,2,3,4$) in sequence. Finally, a neighbour of $\phi$ in the subspace $G^\delta\setminus G^\delta_\mathbb{R}$ is given as
$$U_\phi:=U_1\cap (U_2)^* \cap R_{q^{-1}}(U_3) \cap R_{q^{-1}}((U_4)^*),$$
and by its construction there is a quaternion $q_\phi$ such that $U_\phi$, $(U_\phi)^*$, $R_{q_\phi}(U_\phi)$ and $R_{q_\phi}((U_\phi)^*)$ are disjoint from each other.

Now we proceed to construct the desired measure $\nu$. Let $\mu^{\hat *}=\mu_1+\mu_2$ be the Lebesgue decomposition of $\mu^{\hat *}$ into the absolutely continuous part $\mu_1$ and the singular part $\mu_2$ with respect to $\mu$. According to the premise, the absolutely continuous part does not vanish on $G^\delta\setminus G^\delta_\mathbb{R}$, namely $\mu_1(G^\delta\setminus G^\delta_\mathbb{R})$ is strictly positive. Thus the Radon-Nikodym derivative $d \mu_1/ d \mu$ is not essentially zero on $G^\delta\setminus G^\delta_\mathbb{R}$. Neither is the nonnegative measurable function $f:=\min\left\{1,\frac{d \mu_1}{d \mu}\right\}$, which  implies that $f$ must be not essentially zero on some neighbour $U_\phi$ (w.r.t. $\mu$), since all the neighbours cover the whole subspace $G^\delta\setminus G^\delta_\mathbb{R}$. This observation allows we define a nonzero measure $\nu$ as
$$\nu:=-\lambda-\lambda^{\hat *}+\hat R_{q_\phi}\lambda+ \hat R_{q_\phi}\lambda^{\hat *}. $$
Here $\lambda$ is given by $d\lambda=\left(\chi_{U_\phi}f\right)d\mu$, and $\chi_{U_\phi}$ stands for the characteristic function of the neighbour $U_\phi$. Indeed, $\lambda$ and the three push-forward measures $\neq 0$, because $f$ is not essentially zero on $U_\phi$ with respect to $\mu$. Furthermore, the masses of them lie in four disjoint neighbours $U_\phi$, $(U_\phi)^*$,  $R_{q_\phi}(U_\phi)$ and $R_{q_\phi}((U_\phi)^*)$ respectively, meaning these four nonzero measures are singular. It confirms that $\nu$ is nonzero. In particular, the lower variation of $\nu$ is exactly the sum of $\lambda$ and $\lambda^{\hat *}$.

It remains to show that $\nu$ satisfies the following two conditions.
\begin{itemize}
  \item[i)]$\nu$ belongs to $\mathcal{K}(G^\delta)$, namely $\displaystyle{\int_{G^\delta}\gamma d\nu(\gamma)=0}$.
  \item[ii)]The lower variation of $\nu$, i.e., $\lambda+\lambda^{\hat *}$, is absolutely continuous with respect to $\mu$, and the corresponding Radon-Nikodym derivative is essentially bounded.
\end{itemize}

We start with the first condition. By definition,
\begin{equation}\label{eq-them-p-d-f-G-1}
  \begin{array}{rll}
    \displaystyle{\int_{G^\delta}\gamma d\nu(\gamma)}= &\displaystyle{-\int_{G^\delta}\gamma d\lambda (\gamma)} & \displaystyle{-\int_{G^\delta}\gamma d\lambda^{\hat *}(\gamma)} \\
    & \\
      & \displaystyle{+\int_{G^\delta}\gamma d\hat R_{q_\phi}\lambda (\gamma)} & \displaystyle{+\int_{G^\delta}\gamma d\hat R_{q_\phi}\lambda^{\hat *}(\gamma)}.
  \end{array}
\end{equation}
Here $\gamma$ is considered as an identity function or a vector variable on $G^\delta$.
For simplicity, we write the four terms on the right as $T_i$ ($i=1,\cdots,4$).
Recall that $\hat{*}$ and $\hat R_{q_\phi}$ are the respective push-forward mappings induced by the involution $*$ and the rotation $R_{q_\phi}$ on $G_\delta$. Thus for the last three terms on the right, we have
\begin{align*}
  T_2=&\int_{G^\delta}\gamma^* d\lambda (\gamma), \\
  T_3=&\int_{G^\delta}R_{q_\phi}\gamma d\lambda (\gamma), \\
  T_4=&\int_{G^\delta}R_{q_\phi}\gamma^* d\lambda (\gamma).  \\
\end{align*}
Substituting the three expressions back into \eqref{eq-them-p-d-f-G-1} gives
\begin{equation*}
\int_{G^\delta}\gamma d\nu(\gamma)=\int_{G^\delta}-(\gamma+\gamma^*)+R_{q_\phi}(\gamma+\gamma^*)d\lambda(\gamma).
\end{equation*}
It follows that $\nu$ belongs to $\mathcal{K}(G^\delta)$, since $\gamma+\gamma^*$ is real-valued as a function on $G$ and every real-valued function is invariant under the rotation $R_{q_\phi}$.

Now we turn to the second condition. Recall that $\mu^{\hat *}=\mu_1+\mu_2$ is the Lebesgue decomposition of $\mu^{\hat *}$ into the absolutely continuous part $\mu_1$ and the singular part $\mu_2$ with respect to $\mu$, and $\lambda$ is defined by $d\lambda=\left(\chi_{U_\phi}\min\left\{1,\frac{d \mu_1}{d \mu}\right\}\right)d\mu$, where the neighbour $U_\phi$ is disjoint from its image through the involution, i.e., $(U_\phi)^*$. According to the definition of $\lambda$, we see:
\begin{itemize}
  \item[a.]The mass of $\lambda$ is restricted in $U_\phi$,
  \item[b.]$\lambda\leq\mu_1$ and $\mu$.
\end{itemize}
In addition, applying the involution to the Lebesgue decomposition of $\mu^{\hat *}$ leads to $\mu=\mu_1^{\hat *}+\mu_2^{\hat *}\geq \mu_1^{\hat *}$, which together with Condition b indicates $\mu\geq \lambda^{\hat *}$. Moreover, by Condition b, we see that the mass of $\lambda^{\hat *}$ is restricted in $(U_\phi)^*$. To summarize what we have obtained, firstly, $\lambda$ and $\lambda^{\hat *}$ both $\leq \mu$; secondly, $\lambda$ and $\lambda^{\hat *}$ are mutually singular, since their masses are separately restricted in two disjoint sets. In conclusion, $\lambda+\lambda^{\hat *}\leq\mu$, so that $\lambda+\lambda^{\hat *}$ is absolutely continuous with respect to $\mu$ and the corresponding  Radon-Nikodym derivative is essentially bounded. This completes the proof.
\end{proof}

The answer to the second main question is given in the theorem below.
\begin{theorem}\label{thm-not-Bauer-G}
$\mathcal{P}_*^\mathbb{H}(G)$ is a Bauer simplex
if and only if $G$ is an abelian group of exponent $\leq 2$, or equivalently a $\mathbb{Z}_2$-vector space.
\end{theorem}
Roughly speaking, $\mathcal{P}_*^\mathbb{H}(G)$ is not a Bauer simplex in most cases.
\begin{proof}
The sufficiency is an immediate result of Lemma \ref{lemma-equiv-algebra-feature}. Indeed, when $G$ is of exponent $\leq 2$, the quaternionic dual $G^\delta$, i.e., the extreme boundary of $\mathcal{P}_*^\mathbb{H}(G)$, coincides with the classical dual group $\widehat{G}:=Hom(G,\mathbb{T})$, which is the extreme boundary of a Bauer simplex as is well known.

The necessity can be proved in two ways as follows.

1. If $G$ is of exponent $> 2$, then by Lemma \ref{lemma-equiv-algebra-feature} we see $G^\delta_\mathbb{R}$ is a proper subset of $G^\delta$. Thus there exists a nonnegative Radon measure $\nu\in\mathcal{M}(G^\delta)$ satisfying $\nu(G^\delta\setminus G^\delta_\mathbb{R})>0$.
It is obvious that the nonnegative Radon measure $\mu$ given as $\mu:=\nu+\nu^{\hat *}$ is identical with its involution $\mu^{\hat *}$, so that the restrictions of $\mu$ and $\mu^{\hat *}$ on $G^\delta\setminus G^\delta_\mathbb{R}$ are not mutually singular. Hence, in view of Theorem \ref{thm-property-p-d-f-G}, the uniqueness of the solution to \eqref{Eq-thm-characterization-on-G-integral-expression} breaks down for a certain positive definite function. Therefore, we are led to the conclusion that $\mathcal{P}_*^\mathbb{H}(G)$ is not a Bauer simplex, when $G$ is of exponent $> 2$.

2. Under the same assumption on $G$ as above, we have $G^\delta\setminus G^\delta_\mathbb{R}\neq\emptyset$. Take $\phi$ as an element in $G^\delta\setminus G^\delta_\mathbb{R}$. As we demonstrated in the preceding subsection, the range of $\phi$ is included by some complex slice $C_I$. Then we define $\varphi:=(1+J)\phi(1+J)^{-1}$ where $J$ is an imaginary unit that anti-commutes with $I$. Notice that $\phi$, $\phi^*$, $\varphi$ and $\varphi^*$ are four distinct elements in $G^\delta\setminus G^\delta_\mathbb{R}$. Hence, the sum of the dirac measures centred on the elements $\phi$ and $\phi^*$ does not equals that of the dirac measures centred on the elements $\varphi$ and $\varphi^*$. A direct calculation yields that both the sums correspond to the same positive definite function $2\Re \phi$ via \eqref{Eq-thm-characterization-on-G-integral-expression}. As a consequence, $\mathcal{P}_*^\mathbb{H}(G)$ is not a Bauer simplex.
The proof is now completed.
\end{proof}

\end{document}